\documentclass[11pt,leqno]{amsart}

\usepackage[a4paper, right=20mm, left=20mm, top=25mm]{geometry}

\usepackage{microtype}

\usepackage{amsthm}
\usepackage[english]{babel}
\usepackage[utf8]{inputenc}
\usepackage{hyperref}
\usepackage[foot]{amsaddr}

\usepackage{amsmath}
\usepackage{amssymb}
\usepackage{mathtools}
\usepackage{mathrsfs}

\numberwithin{equation}{section}
\numberwithin{figure}{section}

\theoremstyle{plain}
\newtheorem{Theorem}{Theorem}[section]
\newtheorem*{Theorem*}{Theorem}
\newtheorem{Lemma}[Theorem]{Lemma}

\theoremstyle{definition}
\newtheorem{Definition}[Theorem]{Definition}
\newtheorem{Remark}[Theorem]{Remark}

\newcommand{\C}{\mathbb{C}}
\newcommand{\Z}{\mathbb{Z}}
\newcommand{\kk}{\Bbbk}

\newcommand{\Hom}{\mathrm{Hom}}
\newcommand{\End}{\mathrm{End}}
\newcommand{\Ext}{\mathrm{Ext}}

\newcommand{\Cdot}{\boldsymbol{\cdot}}
\newcommand{\Stab}{\mathrm{Stab}}
\newcommand{\GL}{\mathrm{GL}}
\newcommand{\sign}{\mathrm{sign}}

\title{Cohomology in singular blocks of parabolic category $\mathcal{O}$}
\author{Jonathan Gruber}
\address{Department of Mathematics, National University of Singapore, Singapore}
\email{jgruber@nus.edu.sg}
\subjclass{17B55 (primary), 17B10, 17B67 (secondary)}
\keywords{cohomology, Lie algebra, Kac-Moody algebra, Koszul duality}
\date{\today}

\begin{document}

\begin{abstract}
	We determine the dimensions of $\Ext$-groups between simple modules and dual generalized Verma modules in singular blocks of parabolic versions of category $\mathcal{O}$ for complex semisimple Lie algebras and affine Kac-Moody algebras.
\end{abstract}

\maketitle

\section*{Introduction}
\renewcommand{\theequation}{\Alph{equation}}

Let $\mathfrak{g}$ be a complex semisimple Lie algebra with Borel subalgebra $\mathfrak{b}$ and Cartan subalgebra $\mathfrak{h} \subseteq \mathfrak{b}$, and let $W \subseteq \GL(\mathfrak{h}^*)$ be the Weyl group of $\mathfrak{g}$.
For $\lambda \in \mathfrak{h}^*$, let us write $\nabla_\lambda$ for the dual Verma module of highest weight $\lambda$ and $L_\lambda$ for its unique simple submodule. 
The BGG category $\mathcal{O}$ decomposes into blocks according to the dot action of $W$ on $\mathfrak{h}^*$ given by $w \Cdot \lambda = w(\lambda+\rho) - \rho$ for $w \in W$ and $\lambda \in \mathfrak{h}^*$, where $\rho$ denotes the half-sum of all positive roots of $\mathfrak{g}$ with respect to $\mathfrak{b}$.
Let us write $\mathcal{O}_\lambda$ for the block corresponding to an integral and strictly anti-dominant weight $\lambda \in \mathfrak{h}^*$.
If $\lambda$ is regular (i.e.\ if the stabilizer of $\lambda$ with respect to the dot action of $W$ is trivial) then the composition multiplicities of simple modules in Verma modules in $\mathcal{O}_\lambda$ can be computed as the values at $1$ of certain Kazhdan-Lusztig polynomials.
This result is known as the Kazhdan-Lusztig character formula; it was conjectured by D.\ Kazhdan and G.\ Lusztig in \cite{KazhdanLusztig} and proven by J.-L.\ Brylinski and M.\ Kashiwara in \cite{BrylinskiKashiwara} and by A.\ Be\u{\i}linson and J.\ Bernstein in \cite{BeilinsonBernstein}.
Even before the validity of the Kazhdan-Lusztig character formula had been established, it had been shown by D.\ Vogan in \cite{VoganExt} that its validity is equivalent to the formula
\begin{equation} \label{eq:VoganExt}
	\sum_{i \geq 0} \dim \Ext_{\mathcal{O}}^i\big( L_{x\Cdot\lambda} , \nabla_{y\Cdot\lambda} \big) \cdot v^i = h_{y,x}
\end{equation}
for all $x,y \in W$, where $h_{y,x}$ denotes a Kazhdan-Lusztig polynomial (in the normalization of \cite{SoergelKL}).
In other words, the dimensions of $\Ext$-groups between simple modules and dual Verma modules in $\mathcal{O}_\lambda$ are given by coefficients of Kazhdan-Lusztig polynomials.
The formula \eqref{eq:VoganExt} has been generalized by W.\ Soergel \cite{Soergelncohomology} in two directions:
Firstly, one can replace the regular weight $\lambda$ by a singular weight (i.e.\ by a weight whose stabilizer with respect to the dot action of $W$ is non-trivial), and secondly, one can replace the category $\mathcal{O}_\lambda$ by its parabolic version $\mathcal{O}_\lambda^{\mathfrak{p}}$ (and replace $\nabla_\lambda$ by a dual generalized Verma module $\nabla^\mathfrak{p}_\lambda$) for a parabolic subalgebra $\mathfrak{p}$ of $\mathfrak{g}$ containing $\mathfrak{b}$.
In both cases, the dimensions of $\Ext$-groups between simple $\mathfrak{g}$-modules and dual (generalized) Verma modules are the coefficients of certain parabolic Kazhdan-Lusztig polynomials.
In this note, we further generalize the formula \eqref{eq:VoganExt} by considering blocks corresponding to singular weights in parabolic versions of $\mathcal{O}$.
Our first main result is as follows; see Theorem \ref{thm:extcategoryO}.
The notation is explained in Appendix \ref{sec:appendix}.

\begin{Theorem*}
	Let $\mu \in \mathfrak{h}^*$ be an integral strictly anti-dominant weight with stabilizer $\Stab_W(\mu) = W_I = \langle I \rangle$ for some subset $I$ of the set of simple reflections in $W$, and let $\mathfrak{p} = \mathfrak{p}_J$ be a parabolic subalgebra of $\mathfrak{g}$ containing $\mathfrak{b}$ such that the Levi factor of $\mathfrak{p}$ has Weyl group $W_J = \langle J \rangle$ for a subset $J$ of the set of simple reflections in $W$.
	Then we have
	\[ \sum_{i \geq 0} \dim \Ext_{\mathcal{O}^\mathfrak{p}}^i\big( L_{x\Cdot\lambda} , \nabla^\mathfrak{p}_{y\Cdot\lambda} \big) \cdot v^i = n^I_{y^{-1},x^{-1}} \]
	for all $x,y \in w_J \prescript{J}{}{}{W}^I_\mathrm{reg}$, where $w_J$ denotes the longest element of $W_J$, $\prescript{J}{}{}{W}^I_\mathrm{reg}$ denotes the set of minimal length representatives for the regular double cosets in $W_J \backslash W / W_I$, and $n^I_{y^{-1},x^{-1}}$ denotes an anti-spherical Kazhdan-Lusztig polynomial.
\end{Theorem*}

In Theorem \ref{thm:extKacMoody}, we establish an analogous result for singular blocks of parabolic category $\mathcal{O}$ for an affine Kac-Moody algebra.
The strategy for proving these two theorems is the same:
We first observe that the blocks of parabolic category $\mathcal{O}$ (or certain truncations thereof) admit standard Koszul gradings by the results of \cite{BackelinKoszulduality} and \cite{ShanVaragnoloVasserot}, so that the polynomials that record the dimensions of $\Ext$-groups (as in the theorem above) can be related to certain graded composition multiplicities.
Then we use graded analogues of translation functors in order to deduce our formulas from the known results about dimensions of $\Ext$-groups in regular blocks of parabolic category $\mathcal{O}$.
For affine Kac-Moody algebras, our proof also uses a `double parabolic inversion formula' for parabolic Kazhdan-Lusztig polynomials with respect to two different parabolic subgroups of a given Coxeter group, which we prove in Appendix \ref{sec:appendix}.

\subsection*{Acknowledgments}

The author would like to thank Donna Testerman and Wolfgang Soergel for helpful discussions and comments.
This work was funded by the Swiss National Science Foundation under the grant FNS 200020\_207730 and by the Singapore MOE grant R-146-000-294-133.

\section{Gradings and Koszul duality} \label{sec:Koszul}
\renewcommand{\theequation}{\thesection.\arabic{equation}}

We first recall some well-known results about standard Koszul graded algebras and Koszul duality.

Let $\hat A = \bigoplus_{i \in \Z} A_i$ be a finite-dimensional graded algebra over a field $\kk$, with underlying (ungraded) algebra $A$, and let $\mathcal{C}=A\text{-}\mathrm{mod}$ and $\mathcal{C}_\mathrm{gr} = \hat A\text{-}\mathrm{grmod}$ be the categories of finite-dimensional left $A$-modules and of finite-dimensional graded left $\hat A$-modules, respectively.
In the following, we simply refer to the objects of $\mathcal{C}$ as $A$-modules and to the objects of $\mathcal{C}_\mathrm{gr}$ as graded $\hat A$-modules.
A graded $\hat A$-module $M = \bigoplus_i M_i$ is called \emph{pure} of degree $d \in \Z$ if $M=M_d$.
For $m \in \Z$, let us further denote by $\langle m \rangle \colon \mathcal{C}_\mathrm{gr} \to \mathcal{C}_\mathrm{gr}$ the $m$-th grading shift functor, which sends a graded $\hat A$-module $M = \bigoplus_i M_i$ to the graded $\hat A$-module $M\langle m \rangle = M$, with the grading defined by $M\langle m \rangle_i = M_{i-m}$, and let $\mathrm{f} \colon \mathcal{C}_\mathrm{gr} \to \mathcal{C}$ be the functor that sends a graded $\hat A$-module to the underlying (ungraded) $A$-module.
We call a graded $\hat A$-module $\tilde M$ a \emph{graded lift} of an $A$-module $M$ if $M \cong \mathrm{f}(\tilde M)$.
Now suppose additionally that $\hat A$ is positively graded, i.e.\ that $A_i = 0$ for all $i<0$, and that $A_0$ is a semisimple $\kk$-algebra.
Observe that $A_{>i} = \bigoplus_{j>i} A_j$ is an ideal in $A$ for any $i \geq 0$ and that we can consider $A_0$ as an $A$-module via the identification $A_0 \cong A / A_{>0}$.
\begin{Definition}
	We say that $\hat A$ is \emph{Koszul} if for all $i \in \Z_{\geq 0}$ and $j \in \Z$ with $i \neq j$, we have
	\[ \Ext_{\mathcal{C}_\mathrm{gr}}^i\big( A_0 , A_0 \langle j \rangle \big) = 0 . \]
	In that case, we also say that \emph{$A$ admits a Koszul grading}.
\end{Definition}

Consider the $\Ext$-algebra
\[ E(\hat A) \coloneqq \bigoplus_{i \in \Z} \Ext_\mathcal{C}^i(A_0,A_0) , \]
with the natural grading defined by $E(\hat A)_i = \Ext_\mathcal{C}^i(A_0,A_0)$ for $i \in \Z$.
If $\hat A$ is Koszul then so is $E(\hat A)$, and there is a canonical isomorphism of graded algebras $\hat A \cong E(E(\hat A))$ by Proposition 2.9.1 and Theorems 2.10.1 and 2.10.2 in \cite{BeilinsonGinzburgSoergel}.
When $\hat{A}$ is Koszul, we call the graded algebra $\hat A^! = E(\hat A)^\mathrm{op}$ the \emph{Koszul dual} of $\hat{A}$; it is Koszul because the opposite algebra of a Koszul graded algebra is Koszul by Proposition 2.2.1 in \cite{BeilinsonGinzburgSoergel}.
We also write $A^!$ for the underlying (ungraded) algebra of $\hat A^!$.%
\footnote{According to Corollary 2.5.2 in \cite{BeilinsonGinzburgSoergel}, any two Koszul gradings on $A$ give rise to isomorphic graded algebras.
Therefore, the algebra $A^!$ does not depend on the choice of Koszul grading of $A$.}
Now additionally suppose that $A$ has finite global dimension, so that $E(A)$ is finite-dimensional.
Then according to Theorems 2.12.5 and 2.12.6 in \cite{BeilinsonGinzburgSoergel}, there is an equivalence of triangulated categories
\[ K \colon D^b( \mathcal{C}_\mathrm{gr} ) \longrightarrow D^b( \hat A^!\text{-}\mathrm{grmod} ) \]
such that $K( M \langle i \rangle ) \cong K(M) \langle -i \rangle [ -i ]$ for every graded $\hat A$-module $M$.
\medskip

Now suppose that $\mathcal{C}$ is a highest weight category (see Definition 3.1 in \cite{CPSHighestWeight}) with finite weight poset $(\Lambda,\leq)$ and with simple objects $L_\lambda$, standard objects $\Delta_\lambda$ and costandard objects $\nabla_\lambda$, for $\lambda \in \Lambda$.
Let us further write $I_\lambda$ and $P_\lambda$ for the injective hull and the projective cover of $L_\lambda$, let $P = \bigoplus_{\lambda \in \Lambda} P_\lambda$ be a projective generator of $\mathcal{C}$ and set $A \coloneqq \End_\mathcal{C}(P)^\mathrm{op}$, so that $\mathcal{C}$ is equivalent to $A\text{-}\mathrm{mod}$ via the functor $\Hom_\mathcal{C}(P,-)$.
In the following, we suppress the functor $\Hom_\mathcal{C}(P,-)$ from the notation and simply treat objects of $\mathcal{C}$ as $A$-modules.
From now on, additionally suppose that $A$ admits a positive grading $\hat{A} = \bigoplus_{i \in \Z} A_i$ such that $A_0$ is a semisimple $\kk$-algebra.
Then, up to an isomorphism of graded $\hat A$-modules, any simple object $L_\lambda$ of highest weight $\lambda \in \Lambda$ in $\mathcal{C}$ admits a unique graded lift $\hat L_\lambda$ that is pure of degree $0$, and the objects $\Delta_\lambda$, $\nabla_\lambda$, $P_\lambda$ and $I_\lambda$ of $\mathcal{C}$ admit unique graded lifts $\hat \Delta_\lambda$, $\hat \nabla_\lambda$, $\hat P_\lambda$ and $\hat I_\lambda$ such that all non-zero homomorphisms $\Delta_\lambda \to L_\lambda$, $L_\lambda \to \nabla_\lambda$, $P_\lambda \to L_\lambda$ and $L_\lambda \to I_\lambda$ are homomorphisms of graded $\hat A$-modules.
(See Corollary 4 and the introduction to Section 5 in \cite{MazorchukOvsienkoPairing} for the existence and Lemma 2.5.1 in \cite{BeilinsonGinzburgSoergel} for the uniqueness of these graded lifts.)

\begin{Definition}
	We say that $\hat A$ is \emph{standard Koszul} if for all $\lambda,\mu \in \Lambda$ and all $i \in \Z_{\geq 0}$ and $j \in \Z$ with $i \neq j$, we have
	\[ \Ext_{\mathcal{C}_\mathrm{gr}}^i\big( \hat \Delta_\lambda , \hat L_\mu \langle j \rangle \big) = 0  \qquad \text{and} \qquad \Ext_{\mathcal{C}_\mathrm{gr}}^i\big( \hat L_\mu \langle -j \rangle , \hat \nabla_\mu \big) = 0 . \]
	In that case, we also say that \emph{$A$ admits a standard Koszul grading}.
\end{Definition}

If $\hat{A}$ is standard Koszul then $\hat{A}$ is Koszul, $\hat{A}^!$ is standard Koszul and the category $\mathcal{C}^! = A^!\text{-}\mathrm{mod}$ is a highest weight category with weight poset $(\Lambda,\leq^\mathrm{op})$ by Theorems 1, 2 and 3 in \cite{AgostonDlabLukacsQHextensionalgebra}.
We denote the simple, standard, costandard, indecomposable projective and indecomposable injective object of $\mathcal{C}^!$ of highest weight $\lambda \in \Lambda$ by
\[ L_\lambda^! , \qquad \Delta_\lambda^! , \qquad \nabla_\lambda^! , \qquad P_\lambda^! , \qquad I_\lambda^! , \]
respectively, and we write
\[ \hat L_\lambda^! , \qquad \hat \Delta_\lambda^! , \qquad \hat \nabla_\lambda^! , \qquad \hat P_\lambda^! , \qquad \hat I_\lambda^! \]
for their canonical graded lifts in the category $\mathcal{C}^!_\mathrm{gr} =\hat A^!\text{-}\mathrm{grmod}$.

Now we are ready to explain how Koszul duality can be used to relate dimensions of $\Ext$-groups with graded composition multiplicities.
For $\nu,\lambda \in \Lambda$, consider the polynomials
\[ p_{\nu,\lambda} = \sum_{i \geq 0} \dim \Ext_\mathcal{C}^i( L_\lambda , \nabla_\nu ) \cdot v^i \qquad \text{and} \qquad q_{\nu,\lambda} = \sum_{i \in \Z} [ \hat \Delta_\nu : \hat L_\lambda \langle i \rangle ] \cdot v^i . \]

\begin{Lemma} \label{lem:abstractinversionformula}
	Suppose that $\hat A$ is standard Koszul.
	Then
	\[ \sum_{\nu \in \Lambda} p_{\nu,\lambda}(-v) \cdot q_{\nu,\mu} = \delta_{\lambda,\mu} \]
	for all $\lambda,\mu \in \Lambda$.
\end{Lemma}
\begin{proof}
	The proof is completely analogous to the proof of part (iii) of Theorem 3.11.4 in \cite{BeilinsonGinzburgSoergel}.
\end{proof}

Now for $\nu,\lambda \in \Lambda$, let us additionally consider the polynomials
\[ p_{\nu,\lambda}^! = \sum_{i \geq 0} \dim \Ext_{\mathcal{C}^!}^i( L_\lambda^! , \nabla_\nu^! ) \cdot v^i \qquad \text{and} \qquad q_{\nu,\lambda}^! = \sum_{i \in \Z} [ \hat \Delta_\nu^! : \hat L_\lambda^! \langle i \rangle ] \cdot v^i . \]

\begin{Lemma} \label{lem:KoszuldualKLpolynomials}
	Suppose that $\hat A$ is standard Koszul.
	Then
	\[ q_{\nu,\lambda}^! = p_{\nu,\lambda} \qquad \text{and} \qquad p_{\nu,\lambda}^! = q_{\nu,\lambda}  \]
	for all $\nu,\lambda \in \Lambda$.
\end{Lemma}
\begin{proof}
	The proof is completely analogous to the proof of part (i) of Theorem 3.11.4 in \cite{BeilinsonGinzburgSoergel}, using the fact that the Koszul duality functor satisfies
	\begin{equation*}
	K( \hat L_\lambda ) \cong \hat P_\lambda^! , \qquad K( \hat \nabla_\lambda ) \cong \hat \Delta_\lambda^! , \qquad K( \hat I_\lambda ) \cong \hat L_\lambda^!
	\end{equation*}
	for all $\lambda \in \Lambda$ by Proposition 2.7 in \cite{AgostonDlabLukacsQHextensionalgebra} and Theorem 2.12.5 in \cite{BeilinsonGinzburgSoergel}.
\end{proof}

\section{Complex semisimple Lie algebras} \label{sec:categoryO}

Let $\mathfrak{g}$ be a complex semisimple Lie algebra with a Borel subalgebra $\mathfrak{b}$ and a Cartan subalgebra $\mathfrak{h} \subseteq \mathfrak{b}$, and let $\Phi \subseteq \mathfrak{h}^*$ be the root system of $\mathfrak{g}$ with respect to $\mathfrak{h}$.
We write $\Phi^+ \subseteq \Phi$ for the positive system corresponding to $\mathfrak{b}$ and let $\Pi \subseteq \Phi^+$ be a base of $\Phi$.
The Weyl group
$W = \langle s_\alpha \mid \alpha \in \Phi \rangle \subseteq \GL(\mathfrak{h}^*)$
is a finite Coxeter group with simple reflections $S = \{ s_\alpha \mid \alpha \in \Pi \}$, and we denote by $w_0$ the longest element of $W$.
We write $\alpha^\vee \in \mathfrak{h}$ for the coroot of $\alpha \in \Phi$ and let
\[ X = \{ \lambda \in \mathfrak{h}^* \mid \lambda(\alpha^\vee) \in \Z \text{ for all } \alpha \in \Phi \} \qquad \text{and} \qquad X^+ = \{ \lambda \in X \mid \lambda(\alpha^\vee) \geq 0 \text{ for all } \alpha \in \Phi^+ \} \]
be the sets of \emph{integral weights} and \emph{dominant weights}, respectively.
Furthermore, we write $\rho = \frac12 \cdot \sum_{\alpha \in \Phi^+} \alpha$ for the half-sum of all positive roots and consider the \emph{dot action} of $W$ on $\mathfrak{h}^*$, which is defined by
\[ w \Cdot \lambda = w( \lambda + \rho ) - \rho \]
for $w \in W$ and $\lambda \in \mathfrak{h}^*$.

We consider the category $\mathcal{O}$ of finitely generated $\mathfrak{g}$-modules that are locally $\mathfrak{b}$-finite and admit a weight space decomposition with respect to $\mathfrak{h}$; see for instance \cite{HumphreysCategoryO}.
For all $\lambda \in \mathfrak{h}^*$, the Verma module $\Delta_\lambda$ of highest weight $\lambda$ and its unique simple quotient $L_\lambda$ are objects of $\mathcal{O}$, and so is the dual Verma module $\nabla_\lambda$.
For every weight $\lambda \in - \rho - X^+ = \{ - \rho - \lambda \mid \lambda \in X^+ \}$, we denote by $\mathcal{O}_\lambda$ the block of $\mathcal{O}$ containing the simple $\mathfrak{g}$-module $L_\lambda$.
The stabilizer of $\lambda$ with respect to the dot action of $W$ on $\mathfrak{h}^*$ is a parabolic subgroup of $W$, and for $I \subseteq S$ such that $\Stab_W(\lambda) = W_I = \langle I \rangle$, the isomorphism classes of simple $\mathfrak{g}$-modules in $\mathcal{O}_\lambda$ are in bijection with the set $W^I$ of elements $w \in W$ that have minimal length in the coset $w W_I$ (see Appendix \ref{sec:appendix}), via $w \mapsto L_{w\Cdot\lambda}$.

Now let $J \subseteq S$ and denote by $\mathfrak{p}_J$ the parabolic subalgebra of $\mathfrak{g}$ that is generated by $\mathfrak{b}$ together with the root spaces $\mathfrak{g}_{-\alpha}$ for $\alpha \in \Pi$ with $s_\alpha \in J$.
We write $\mathcal{O}^{\mathfrak{p}_J}$ for the full subcategory of $\mathcal{O}$ whose objects are the locally $\mathfrak{p}_J$-finite $\mathfrak{g}$-modules in $\mathcal{O}$, and for $\lambda \in -\rho - X^+$, let $\mathcal{O}_\lambda^{\mathfrak{p}_J}$ be the full subcategory of $\mathcal{O}_\lambda$ whose objects are the locally $\mathfrak{p}_J$-finite $\mathfrak{g}$-modules in $\mathcal{O}_\lambda$.
For $I \subseteq S$ such that $\Stab_W(\lambda) = W_I$ and for $w \in W^I$, the simple $\mathfrak{g}$-module $L_{w\Cdot\lambda}$ belongs to $\mathcal{O}_\lambda^{\mathfrak{p}_J}$ if and only if $(w\Cdot\lambda)(\alpha^\vee) \geq 0$ for all $\alpha \in J$, or equivalently, if $w \in w_J \prescript{J}{}{W} \cap W^I = w_J \prescript{J}{}{W}^I_\mathrm{reg}$.
Here, we write $w_J$ for the longest element of $W_J$ and $\prescript{J}{}{W}$ for the set of elements $w \in W$ that have minimal length in the coset $W_J w$, and we denote by $\prescript{J}{}{W}^I_\mathrm{reg}$ the set of minimal length representatives for the regular double cosets in $W_J \backslash W / W_I$ (see Appendix \ref{sec:appendix}).
In particular, the isomorphism classes of simple $\mathfrak{g}$-modules in $\mathcal{O}_\lambda^{\mathfrak{p}_J}$ are in bijection with the set $w_J \prescript{J}{}{W}^I_\mathrm{reg}$ via $w \mapsto L_{w\Cdot\lambda}$.
For $\mu \in \mathfrak{h}^*$, let us write $\Delta^{\mathfrak{p}_J}_\mu$ for the largest quotient of $\Delta_\mu$ that belongs to $\mathcal{O}^{\mathfrak{p}_J}$ and $\nabla^{\mathfrak{p}_J}_\mu$ for the largest submodule of $\nabla_\mu$ that belongs to $\mathcal{O}^{\mathfrak{p}_J}$, and call these $\mathfrak{g}$-modules the \emph{generalized} or \emph{parabolic} Verma module and dual Verma module, respectively, of highest weight $\mu$ with respect to $\mathfrak{p}_J$.
Then $\mathcal{O}_\lambda^{\mathfrak{p}_J}$ is a highest weight category with weight poset $(w_J \prescript{J}{}{W}^I_\mathrm{reg},\leq)$, where $\leq$ denotes the Bruhat order, with simple objects $L_{w\Cdot\lambda}$, standard objects $\Delta^{\mathfrak{p}_J}_{w\Cdot\lambda}$ and costandard objects $\nabla^{\mathfrak{p}_J}_{w\Cdot\lambda}$ for $w \in w_J \prescript{J}{}{W}^I_\mathrm{reg}$.
Let us further write $P^{\mathfrak{p}_J}_{w\Cdot\lambda}$ for the projective cover of $L_{w\Cdot\lambda}$ in $\mathcal{O}^{\mathfrak{p}_J}_\lambda$, for $w \in w_J \prescript{J}{}{W}^I_\mathrm{reg}$.
We define
\[ \mathbf{P}^{\mathfrak{p}_J}_\lambda \coloneqq \bigoplus_{ w \in w_J \prescript{J}{}{W}^I_\mathrm{reg} } P^{\mathfrak{p}_J}_{w \Cdot \lambda} \qquad \text{and} \qquad A_\lambda^{\mathfrak{p}_J} \coloneqq \End_{\mathcal{O}^{\mathfrak{p}_J}}\big( \mathbf{P}^{\mathfrak{p}_J}_\lambda \big)^\mathrm{op} , \]
so that $\mathcal{O}^{\mathfrak{p}_J}_\lambda$ is equivalent to the category $A^{\mathfrak{p}_J}_\lambda\text{-}\mathrm{mod}$.
The following theorem combines results of E.\ Backelin and V.\ Mazorchuk; it will be crucial for our determination of the dimensions of $\Ext$-groups between simple objects and costandard objects in $\mathcal{O}^{\mathfrak{p}_J}_\lambda$.

\begin{Theorem} \label{thm:categoryOkoszuldual}
	Let $\lambda,\mu \in - \rho - X^+$ and let $I,J \subseteq S$ such that $\Stab_W(\lambda) = W_I$ and $\Stab_W(\mu) = W_J$.
	Then the algebra $A_\lambda^{\mathfrak{p}_J}$ admits a standard Koszul grading, and its Koszul dual $(A_\lambda^{\mathfrak{p}_J})^!$ is isomorphic to the algebra $A_{-w_0\mu}^{\mathfrak{p}_I}$.
\end{Theorem}
\begin{proof}
	By Theorem 1.1 in \cite{BackelinKoszulduality}, the algebra $A_\lambda^{\mathfrak{p}_J}$ admits a Koszul grading and its Koszul dual $(A_\lambda^{\mathfrak{p}_J})^!$ is isomorphic to the algebra $A_{-w_0\mu}^{\mathfrak{p}_I}$.
	The standard Koszulity of the Koszul graded version of $A_\lambda^{\mathfrak{p}_J}$ is established in Theorem 5.1 of \cite{MazorchukApplicationslinearcomplexes}.
\end{proof}

We keep the notation and assumptions of Theorem \ref{thm:categoryOkoszuldual} and define
\[ \mathbf{L}^{\mathfrak{p}_J}_\lambda \coloneqq \bigoplus_{ w \in w_J \prescript{J}{}{W}^I_\mathrm{reg} } L_{w \Cdot \lambda} , \qquad \text{so that} \qquad ( A_\lambda^{\mathfrak{p}_J} )^! \cong \Big( \bigoplus_{i \geq 0} \Ext_{\mathcal{O}^{\mathfrak{p}_J}}^i\big( \mathbf{L}^{\mathfrak{p}_J}_\lambda , \mathbf{L}^{\mathfrak{p}_J}_\lambda \big) \Big)^\mathrm{op} . \]
For $w \in w_J \prescript{J}{}{W}^I_\mathrm{reg}$, we further write $e^{\mathfrak{p}_J}_{\lambda,w}$ for the idempotent in $A^{\mathfrak{p}_J}_\lambda$ corresponding to the canonical projection $\mathrm{P}^{\mathfrak{p}_J}_\lambda \to P^{\mathfrak{p}_J}_{w\Cdot\lambda}$ and $f^{\mathfrak{p}_J}_{\lambda,w}$ for the idempotent in the Koszul dual $( A^{\mathfrak{p}_J}_\lambda )^!$ corresponding to the canonical projection $\mathrm{L}^{\mathfrak{p}_J}_\lambda \to L_{w\Cdot\lambda}$.
Then, according to Remark 3.8 (and Proposition 3.1) in \cite{BackelinKoszulduality}, the isomorphism
\[ \varphi \colon ( A^{\mathfrak{p}_J}_\lambda )^! \longrightarrow A^{\mathfrak{p}_I}_{-w_0\mu} \]
from Theorem \ref{thm:categoryOkoszuldual} can be chosen in such a way that
\[ \varphi\big( f^{\mathfrak{p}_J}_{\lambda,w} \big) = e^{\mathfrak{p}_I}_{-w_0\mu,w^{-1} w_0} \]
for all $w \in w_J \prescript{J}{}{W}^I_\mathrm{reg}$.%
\footnote{Observe that the correspondence between the idempotents in $( A^{\mathfrak{p}_J}_\lambda )^!$ and $A^{\mathfrak{p}_I}_{-w_0\mu}$ that we describe here differs from the one in \cite{BackelinKoszulduality} because we label the blocks of $\mathcal{O}$ by weights in $-\rho - X^+$, rather than $-\rho + X^+$.}
In particular, if we write $L^{\mathfrak{p}_J,!}_{\lambda,w}$ for the simple $( A^{\mathfrak{p}_J}_\lambda )^!$-module corresponding to an element $w \in w_J \prescript{J}{}{W}^I_\mathrm{reg}$ (as described in Section \ref{sec:Koszul}) then the functor
\[ ( \mathcal{O}^{\mathfrak{p}_J}_\lambda )^! \coloneqq ( A^{\mathfrak{p}_J}_\lambda )^!\text{-}\mathrm{mod} \longrightarrow A^{\mathfrak{p}_J}_{-w_0\mu}\text{-}\mathrm{mod} \cong \mathcal{O}^{\mathfrak{p}_J}_{-w_0\mu} \]
induced by $\varphi^{-1}$ sends the simple $( A^{\mathfrak{p}_J}_\lambda )^!$-module $L^{\mathfrak{p}_J,!}_{\lambda,w}$ to the simple $\mathfrak{g}$-module $L_{w^{-1} w_0\Cdot(-w_0\mu)}$.

Now for a subset $I \subseteq S$ and $x,y \in \prescript{I}{}{W}$, let us denote by $m^I_{x,y}$, $m_I^{x,y}$, $n^I_{x,y}$ and $n_I^{x,y}$ the corresponding spherical and anti-spherical (inverse) parabolic Kazhdan-Lusztig polynomials, as defined in Appendix \ref{sec:appendix}.
With all of the above notation in place, we can prove the first main result.

\begin{Theorem} \label{thm:extcategoryO}
	Let $\lambda,\mu \in - \rho - X^+$ and let $I,J \subseteq S$ such that $\Stab_W(\lambda) = W_I$ and $\Stab_W(\mu) = W_J$.
	For $x,z \in w_J \prescript{J}{}{}{W}^I_\mathrm{reg}$, we have
	\[ \sum_{i \geq 0} \dim \Ext_{\mathcal{O}^{\mathfrak{p}_J}}^i\big( L_{x\Cdot\lambda} , \nabla^{\mathfrak{p}_J}_{z\Cdot\lambda} \big) \cdot v^i = n^I_{z^{-1},x^{-1}} . \]
\end{Theorem}
\begin{proof}
	For $w \in w_J \prescript{J}{}{W}^I_\mathrm{reg}$, let us write $\hat L^{\mathfrak{p}_J}_{w\Cdot\lambda}$ and $\hat \Delta^{\mathfrak{p}_J}_{w\Cdot\lambda}$ for the canonical graded lifts of $L_{w\Cdot\lambda}$ and $\Delta^{\mathfrak{p}_J}_{w\Cdot\lambda}$, respectively, in the category of graded modules over a Koszul graded version of $A^{\mathfrak{p}_J}_\lambda$.
	Then, according to Lemma \ref{lem:KoszuldualKLpolynomials} and the above discussion of Koszul duality for $\mathcal{O}^{\mathfrak{p}_J}_\lambda$, the polynomials
	\[ p^{\mathfrak{p}_J}_{\lambda,z,x} = \sum_{i \geq 0} \dim \Ext_{\mathcal{O}^{\mathfrak{p}_J}}^i( L_{x\Cdot\lambda} , \nabla^{\mathfrak{p}_J}_{z\Cdot \lambda} ) \cdot v^i \qquad \text{and} \qquad q^{\mathfrak{p}_J}_{\lambda,z,y} = \sum_{i \geq 0} [ \hat \Delta^{\mathfrak{p}_J}_{z\Cdot \lambda} : \hat L^{\mathfrak{p}_J}_{y\Cdot\lambda} \langle i \rangle ] \cdot v^i \]
	with $x,y,z \in w_J \prescript{J}{}{W}^I_\mathrm{reg}$ are related via the equations
	\begin{equation} \label{eq:KoszuldualKLpolynomialscategoryO}
	p^{\mathfrak{p}_J}_{\lambda,z,x} = q^{\mathfrak{p}_I}_{-w_0\mu,z^{-1}w_0,x^{-1}w_0} \qquad \text{and} \qquad q^{\mathfrak{p}_J}_{\lambda,z,y} = p^{\mathfrak{p}_I}_{-w_0\mu,z^{-1}w_0,y^{-1}w_0} .
	\end{equation}
	We first consider a special case of the formula in the theorem, which we will then use to prove the general statement:
	For $\lambda^\prime \in - \rho - X^+$ with $\Stab_W(\lambda^\prime) = \{ e \}$, we have
	\[ q^{\mathfrak{p}_J}_{\lambda^\prime,z,y} = p^{\mathfrak{b}}_{-w_0\mu,z^{-1}w_0,y^{-1}w_0} = n^J_{zw_0,yw_0} \]
	for all $y,z \in w_J \prescript{J}{}{W}$ by parts (i) and (iv) of Theorem 3.11.4 in \cite{BeilinsonGinzburgSoergel}.
	Now let $\hat A^{\mathfrak{p}_J}_\lambda$ and $\hat A^{\mathfrak{p}_J}_{\lambda^\prime}$ be Koszul graded algebras whose underlying ungraded algebras are $A^{\mathfrak{p}_J}_\lambda$ and $A^{\mathfrak{p}_J}_{\lambda^\prime}$, respectively, and let
	\[ \mathcal{O}^{\mathfrak{p}_J}_{\lambda,\mathrm{gr}} = \hat A^{\mathfrak{p}_J}_\lambda \text{-} \mathrm{mod} \qquad \text{and} \qquad \mathcal{O}^{\mathfrak{p}_J}_{\lambda^\prime,\mathrm{gr}} = \hat A^{\mathfrak{p}_J}_{\lambda^\prime} \text{-} \mathrm{mod} \]
	be the corresponding categories of graded modules.
	For $J=\varnothing$, we omit the superscript $\mathfrak{p}_J$ and write
	\[ \mathcal{O}_{\lambda,\mathrm{gr}} = \hat A_\lambda \text{-} \mathrm{mod} \qquad \text{and} \qquad \mathcal{O}_{\lambda^\prime,\mathrm{gr}} = \hat A_{\lambda^\prime} \text{-} \mathrm{mod} . \]
	Then, by the proof of Proposition 3.2 in \cite{BackelinKoszulduality}, the graded algebra $\hat A^{\mathfrak{p}_J}_\lambda$ is the quotient of $\hat A_\lambda$ by some homogeneous ideal, whence we can consider $\mathcal{O}^{\mathfrak{p}_J}_{\lambda,\mathrm{gr}}$ as a full subcategory of $\mathcal{O}_{\lambda,\mathrm{gr}}$, and similarly for $\mathcal{O}^{\mathfrak{p}_J}_{\lambda^\prime,\mathrm{gr}}$ and $\mathcal{O}_{\lambda^\prime,\mathrm{gr}}$.
	(See also Lemma 2.2 in \cite{ShanVaragnoloVasserot} and Theorem 3.5.3 in \cite{BeilinsonGinzburgSoergel}.)
	 Furthermore, as shown in part (d) of the proof of Lemma 3.3 in \cite{BackelinKoszulduality}, there is an exact functor
	 \[ \hat T_{\lambda^\prime}^\lambda \colon \mathcal{O}_{\lambda^\prime,\mathrm{gr}} \longrightarrow \mathcal{O}_{\lambda,\mathrm{gr}} \]
	 which lifts the usual translation functor $T_{\lambda^\prime}^\lambda \colon \mathcal{O}_{\lambda^\prime} \to \mathcal{O}_\lambda$ and sends any pure object in $\mathcal{O}_{\lambda^\prime,\mathrm{gr}}$ to a pure object of the same degree in $\mathcal{O}_{\lambda,\mathrm{gr}}$.
	 Using well-known properties of translation functors (see Section 7.9 in \cite{HumphreysCategoryO}), it follows that
	 \[ \hat T_{\lambda^\prime}^\lambda \hat L_{y\Cdot\lambda^\prime} \cong \begin{cases} \hat L_{y\Cdot\lambda} & \text{if } y \in W^I , \\ 0 & \text{otherwise} \end{cases} \]
	 for all $y \in W$ and that $\hat T_{\lambda^\prime}^\lambda$ restricts to a functor from $\mathcal{O}^{\mathfrak{p}_J}_{\lambda^\prime,\mathrm{gr}}$ to $\mathcal{O}^{\mathfrak{p}_J}_{\lambda,\mathrm{gr}}$ with
	 \[ \hat T_{\lambda^\prime}^\lambda \hat \Delta^{\mathfrak{p}_J}_{z\Cdot\lambda^\prime} \cong \hat \Delta^{\mathfrak{p}_J}_{z\Cdot\lambda} \]
	 for all $z \in w_J \prescript{J}{}{W}^I_\mathrm{reg}$.%
	 \footnote{This follows from the corresponding statement for Verma modules (see Section 7.6 in \cite{HumphreysCategoryO}) and the facts that $\Delta^{\mathfrak{p}_J}_{z \Cdot \lambda}$ arises from $\Delta_{z \Cdot \lambda}$ by applying a \emph{parabolic truncation functor} (and similarly for $\lambda^\prime$) and that parabolic truncation functors commute with translation functors; see Lemma 2.6 in \cite{BackelinKoszulduality}.
	 See also Section 2.25 in \cite{JantzenHabilitation}.}
	 Now the exactness of $\hat T_{\lambda^\prime}^\lambda$ implies that
	 \[ [ \hat \Delta^{\mathfrak{p}_J}_{z\Cdot\lambda} : \hat L^{\mathfrak{p}_J}_{y\Cdot\lambda} \langle i \rangle ] = [ \hat \Delta^{\mathfrak{p}_J}_{z\Cdot\lambda^\prime} : \hat L^{\mathfrak{p}_J}_{y\Cdot\lambda^\prime} \langle i \rangle ] \]
	for all $y,z \in w_J \prescript{J}{}{W}^I_\mathrm{reg}$ and all $i \in \Z$.
	In particular, the special case at the beginning of the proof implies that $q^{\mathfrak{p}_J}_{\lambda,z,y} = q^{\mathfrak{p}_J}_{\lambda^\prime,z,y} = n^J_{zw_0,yw_0}$ for all $y,z \in w_J \prescript{J}{}{W}^I_\mathrm{reg}$, and using \eqref{eq:KoszuldualKLpolynomialscategoryO}, we obtain
	\[ p^{\mathfrak{p}_J}_{\lambda,z,x} = q^{\mathfrak{p}_I}_{-w_0 \mu,z^{-1}w_0,x^{-1}w_0} = n^I_{z^{-1},x^{-1}} \]
	for all $x,z \in w_J \prescript{J}{}{W}^I_\mathrm{reg}$, as required.
\end{proof}

\section{Affine Kac-Moody algebras} \label{sec:KacMoody}

Let us keep the notation from the previous section and additionally assume that $\mathfrak{g}$ is a complex \emph{simple} Lie algebra.
Recall that we fix a Borel subalgebra $\mathfrak{b}$ and a Cartan subalgebra $\mathfrak{h} \subseteq \mathfrak{b}$ of $\mathfrak{g}$.
We consider the affine Kac-Moody Lie algebra
\[ \mathfrak{\tilde g} = ( \C[t^{\pm1}] \otimes_\C \mathfrak{g} ) \oplus \C c \oplus \C d \]
as in \cite[Section 6]{TanisakiCharacterFormulas} and its subalgebras
\[ \mathfrak{\tilde b} = \mathfrak{b} \oplus ( t \C[t] \otimes \mathfrak{g} ) \oplus \C c \oplus \C d \qquad \text{and} \qquad \mathfrak{\tilde h} = \mathfrak{h} \oplus \C c \oplus \C d , \]
and we view $\mathfrak{h}^*$ as a subspace of $\mathfrak{\tilde h}^*$, with the convention that $\lambda(c) = \lambda(d) = 0$ for all $\lambda \in \mathfrak{h}^*$.
Let $\tilde \Phi$ be the root system of $\mathfrak{\tilde g}$ and let $\tilde \Phi^+$ be the set of positive roots corresponding to $\mathfrak{\tilde b}$, again as in \cite[Section 6]{TanisakiCharacterFormulas}.
Further let $\tilde \Pi = \Pi \sqcup \{ \alpha_0 \}$ be the set of simple roots in $\tilde \Phi$ and let
\[\tilde W = \langle s_\alpha \mid \alpha \in \tilde \Pi \rangle \subseteq \GL(\mathfrak{\tilde h}^*)\]
be the \emph{(affine) Weyl group} of $\mathfrak{\tilde g}$.
It is a Coxeter group with set of simple reflections $\tilde S = \{ s_\alpha \mid \alpha \in \tilde \Pi \}$ and it is canonically isomorphic to the semidirect product $W \rtimes \Z\Phi^\vee$, where $W$ denotes the Weyl group of $\mathfrak{g}$ and $\Z\Phi^\vee$ is the coroot lattice of $\mathfrak{g}$ (see Proposition 13.1.7 in \cite{KumarKacMoody}).
Let us write $\alpha^\vee \in \mathfrak{\tilde h}$ for the coroot corresponding to $\alpha \in \tilde \Pi$ and define $\tilde \rho \in \mathfrak{\tilde h}^*$ by $\tilde \rho(\alpha^\vee) = 1$ for all $\alpha \in \tilde \Pi$ and $\tilde \rho(d) = 0$.
The \emph{dot action} of $\tilde W$ on $\mathfrak{\tilde h}^*$ is given by
\[ w \Cdot \lambda = w( \lambda + \tilde \rho ) - \tilde \rho \]
for all $w \in \tilde W$ and $\lambda \in \mathfrak{\tilde h}^*$.
We say that $\lambda \in \mathfrak{\tilde h}^*$ has \emph{level} $e \in \C$ if $e = ( \lambda + \tilde \rho )(c)$ and we set
\[ \lambda^\prime \coloneqq - \lambda - 2 \tilde \rho , \]
so that $\lambda^\prime$ has level $-e$ when $\lambda$ has level $e$ and $(w \Cdot \lambda)^\prime = w\Cdot\lambda^\prime$ for all $w \in \tilde W$.
The sets of \emph{integral} and \emph{dominant} weights are defined by
\[ \tilde X = \{ \lambda \in \mathfrak{\tilde h}^* \mid \lambda(\alpha^\vee) \in \Z \text{ for all } \alpha \in \tilde \Pi \} \qquad \text{and} \qquad \tilde X^+ = \{ \lambda \in \tilde X \mid \lambda(\alpha^\vee) \geq 0 \text{ for all } \alpha \in \tilde \Pi \} , \]
respectively, and we note that for every integral weight $\mu \in \tilde X$ of positive (or negative) level, there is a unique weight $\nu \in \tilde X^+ \setminus\{0\}$ such that $\mu$ belongs to the $\tilde W$-orbit of $\nu - \tilde \rho$ (respectively $-\nu - \tilde \rho$) with respect to the dot action; see for instance Section 6.3 in \cite{BrundanStroppel}.

Now let us consider the category $\mathcal{\tilde O}$ of $\mathfrak{\tilde g}$-modules that admit a weight space decomposition with finite-dimensional weight spaces with respect to $\mathfrak{\tilde h}$ and whose set of weights is contained in the union of finitely many sets of the form $\{ \mu \in \mathfrak{\tilde h}^* \mid \mu \leq \lambda \}$ for $\lambda \in \mathfrak{\tilde h}^*$, where we write $\leq$ for the partial order on $\mathfrak{\tilde h}^*$ that is defined by
\[ \mu \leq \lambda \qquad \text{if and only if} \qquad \lambda - \mu \in \sum_{\alpha \in \tilde\Pi} \Z_{\geq 0} \cdot \alpha . \]
For all $\lambda \in \mathfrak{\tilde h}^*$, the Verma module $\Delta_\lambda$ of highest weight $\lambda$, the dual Verma module $\nabla_\lambda$ and the unique simple quotient $L_\lambda$ of $\Delta_\lambda$ (which is also the unique simple submodule of $\nabla_\lambda$) are objects of $\mathcal{\tilde O}$.
Even though a $\mathfrak{\tilde g}$-module $M$ in $\mathcal{\tilde O}$ does not need to admit a finite composition series, the composition multiplicities $[ M : L_\mu ]$ are still (well-defined and) finite for $\mu \in \mathfrak{\tilde h}^*$ (see Lemma 2.1.9 in \cite{KumarKacMoody}).
Now let us fix
\[ \lambda \in ( - \tilde \rho + \tilde X^+ ) \cup ( - \tilde \rho - \tilde X^+ ) \qquad \text{with} \qquad \lambda \neq - \tilde \rho \]
and write $\mathcal{\tilde O}_\lambda$ for the full subcategory of $\mathcal{\tilde O}$ whose objects are the $\mathfrak{\tilde g}$-modules $M$ in $\mathcal{\tilde O}$ such that $[ M : L_\mu ] = 0$ for all $\mu \in \mathfrak{\tilde h}^* \setminus \tilde W \Cdot \lambda$.
Furthermore, for $J \subseteq \tilde S$, we write $\mathfrak{p}_J$ for the Levi subalgebra of $\mathfrak{\tilde g}$ corresponding to $J$ and we let $\mathcal{\tilde O}^{\mathfrak{p}_J}$ (or $\mathcal{\tilde O}^{\mathfrak{p}_J}_\lambda$) be the full subcategory of $\mathcal{\tilde O}$ (or $\mathcal{\tilde O}_\lambda$) whose objects are the locally $\mathfrak{p}_J$-finite $\mathfrak{\tilde g}$-modules in $\mathcal{\tilde O}$ (or $\mathcal{\tilde O}_\lambda$).
As in the previous section, we write $\Delta^{\mathfrak{p}_J}_\mu$ for the largest locally $\mathfrak{p}_J$-finite quotient of the Verma module $\Delta_\mu$ of highest weight $\mu \in \mathfrak{\tilde h}^*$ and $\nabla^{\mathfrak{p}_J}_\mu$ for the largest locally $\mathfrak{p}_J$-finite submodule of $\nabla_\mu$, and we call these $\mathfrak{\tilde g}$-modules the \emph{generalized Verma module} and the \emph{dual generalized Verma module} of highest weight $\mu$ with respect to $\mathfrak{p}_J$.

Now let $I \subseteq \tilde S$ such that $\Stab_{\tilde W}(\lambda) = \tilde W_I$.
If $\lambda$ has positive (or negative) level then we parameterize the set of isomorphism classes of simple $\mathfrak{\tilde g}$-modules in $\mathcal{\tilde O}_\lambda$ by $X_\lambda \coloneqq \tilde W^I w_I$ (respectively $X_\lambda \coloneqq \tilde W^I$) via the map $w \mapsto L_{w\Cdot\lambda}$, and this parametrization restricts to a bijection between the set of isomorphism classes of simple $\mathfrak{\tilde g}$-modules in $\mathcal{\tilde O}_\lambda^{\mathfrak{p}_J}$ and the set $X_\lambda^J \coloneqq {}^J \tilde W^I_\mathrm{reg} w_I$ (respectively $X_\lambda^J \coloneqq w_J {}^J \tilde W^I_\mathrm{reg}$) of representatives for the regular double cosets in $\tilde W_J \backslash \tilde W / \tilde W_I$ (as defined in Appendix \ref{sec:appendix}).
In order to apply the considerations from Section \ref{sec:Koszul}, we consider truncations of the category $\mathcal{\tilde O}_\lambda^{\mathfrak{p}_J}$ as follows:
If $\lambda$ has negative level then for $w \in X_\lambda^J$, we write $\mathcal{\tilde O}^{\mathfrak{p}_J}_{\lambda,\leq w}$ for the Serre subcategory of $\mathcal{\tilde O}^{\mathfrak{p}_J}_{\lambda}$ generated by the simple $\mathfrak{\tilde g}$-modules $L_{x\Cdot\lambda}$ with $x \in X_\lambda^{J, \leq w}$, where we set
\[ X_\lambda^{J, \leq w} \coloneqq \{ y \in X_\lambda^J \mid y \leq w \} . \]
Then $\mathcal{\tilde O}^{\mathfrak{p}_J}_{\lambda,\leq w}$ is a highest weight category with weight poset $( X_\lambda^{J, \leq w} , \leq )$ by Lemma 3.7 in \cite{ShanVaragnoloVasserot}, where $\leq$ denotes the Bruhat order and the standard objects and costandard objects are given by $\Delta^{\mathfrak{p}_J}_{x\Cdot\lambda}$ and $\nabla^{\mathfrak{p}_J}_{x\Cdot\lambda}$, respectively, for $x \in X_\lambda^{J, \leq w}$.
If $\lambda$ has positive level then for $w \in X_\lambda^J$, we write $\mathcal{\tilde O}^{\mathfrak{p}_J}_{\lambda,\leq w}$ for the Serre quotient category of $\mathcal{\tilde O}^{\mathfrak{p}_J}_{\lambda}$ by the Serre subcategory generated by the simple $\mathfrak{\tilde g}$-modules $L_{x\Cdot\lambda}$ for $x \in X_\lambda^J$ with $x \nleq w$, so that the simple objects of $\mathcal{\tilde O}^{\mathfrak{p}_J}_{\lambda,\leq w}$ are naturally parameterized by the set
\[ X_\lambda^{J,\leq w} \coloneqq \{ y \in X_\lambda^J \mid y \leq w \} . \]
Then $\mathcal{\tilde O}^{\mathfrak{p}_J}_{\lambda,\leq w}$ is a highest weight category with weight poset $( X_\lambda^{J, \leq w} , \geq )$ by Proposition 3.9 in \cite{ShanVaragnoloVasserot}, where $\geq$ denotes the opposite Bruhat order and the standard objects and costandard objects are given by $\Delta^{\mathfrak{p}_J}_{x\Cdot\lambda}$ and $\nabla^{\mathfrak{p}_J}_{x\Cdot\lambda}$, respectively, for $x \in X_\lambda^{J, \leq w}$.
In either case (i.e.\ for $\lambda$ of positive or negative level), we write $P^{\mathfrak{p}_J,\leq w}_{x\Cdot\lambda}$ for the projective cover of $L_{x\Cdot\lambda}$ in $\mathcal{\tilde O}^{\mathfrak{p}_J}_{\lambda,\leq w}$, for $x \in X_\lambda^{J, \leq w}$, and we set
\[ \mathrm{P}_{\lambda,\leq w}^{\mathfrak{p}_J} \coloneqq \bigoplus_{x \in X_\lambda^{\mathfrak{p}_J,\leq w}} P_{x\Cdot\lambda}^{\mathfrak{p}_J,\leq w} \qquad \text{and} \qquad A_{\lambda,\leq w}^{\mathfrak{p}_J} = \End_{\mathcal{\tilde O}^{\mathfrak{p}_J}_{\lambda,\leq w}}\big( \mathrm{P}_{\lambda,\leq w}^{\mathfrak{p}_J} \big)^\mathrm{op} ,  \]
so that $\mathcal{\tilde O}_{\lambda,\leq w}^{\mathfrak{p}_J}$ is equivalent to $A_{\lambda,\leq w}^{\mathfrak{p}_J}\text{-}\mathrm{mod}$.

\begin{Remark}
	For $\lambda$ of negative level, the fact that all of the truncated categories $\mathcal{\tilde O}_{\lambda,\leq w}^{\mathfrak{p}_J}$ are highest weight categories implies that $\mathcal{\tilde O}_\lambda^{\mathfrak{p}_J}$ is a \emph{lower finite highest weight category} with weight poset $(X_\lambda^J,\leq)$, in the sense of Definition 3.50 in \cite{BrundanStroppel}, where as before, we write $\leq$ for the Bruhat order.
	For $\lambda$ of positive level, the category $\mathcal{\tilde O}_\lambda^{\mathfrak{p}_J}$ is an \emph{upper finite highest weight category} with weight poset $(X_\lambda^J,\geq)$ in the sense of \cite[Definition 3.34]{BrundanStroppel} because every simple $\mathfrak{\tilde g}$-module in $\mathcal{\tilde O}_\lambda^{\mathfrak{p}_J}$ has a projective cover in $\mathcal{\tilde O}_\lambda^{\mathfrak{p}_J}$ which admits a filtration by generalized Verma modules; see \cite[Section 4]{RochaCaridiWallach}.
\end{Remark}

As in the preceding section, we want to use Koszul duality as a tool to compute the dimensions of $\Ext$-groups in $\mathcal{\tilde O}_{\lambda,\leq w}^{\mathfrak{p}_J}$.
This is possible by the following result; see Theorem 3.12 in \cite{ShanVaragnoloVasserot}.

\begin{Theorem} \label{thm:KacMoodykoszuldual}
	Let $I,J \subsetneq \tilde S$ and let $\lambda \in - \tilde \rho - \tilde X^+$ and $\mu \in - \tilde \rho + \tilde X^+$ such that $\Stab_{\tilde W}(\lambda) = \tilde W_I$ and $\Stab_{\tilde W}(\mu) = \tilde W_J$.
	Then, for $w \in X_\lambda^J$ and $u \in X_\mu^I$, the algebras $A_{\lambda,\leq w}^{\mathfrak{p}_J}$ and $A_{\mu,\leq u}^{\mathfrak{p}_I}$ are standard Koszul, and their Koszul duals are given by
	\[ ( A_{\lambda,\leq w}^{\mathfrak{p}_J} )^! \cong A_{\mu,\leq w^{-1}}^{\mathfrak{p}_I} \qquad \text{and} \qquad ( A_{\mu , \leq u}^{\mathfrak{p}_I} )^! \cong A_{\lambda,\leq u^{-1}}^{\mathfrak{p}_J} . \]
\end{Theorem}

Now let us return to a weight $\lambda \in ( - \tilde \rho + \tilde X^+ ) \cup ( - \tilde \rho - \tilde X^+ )$ with $\lambda \neq - \rho$ and subsets $I,J \subsetneq \tilde S$ such that $\Stab_{\tilde W}(\lambda) = \tilde W_I$.
Let us further fix a weight $\mu \in ( - \tilde \rho + \tilde X^+ ) \cup ( - \tilde \rho - \tilde X^+ )$ with $\Stab_{\tilde W}(\mu) = \tilde W_J$ such that $\mu$ has negative (or positive) level if $\lambda$ has positive (or negative) level.
For $w \in X_\lambda^J$, we set
\[ \mathbf{L}^{\mathfrak{p}_J}_{\lambda,\leq w} \coloneqq \bigoplus_{ x \in X_\lambda^{J,\leq w} } L_{x \Cdot \lambda} , \qquad \text{so that} \qquad ( A_{\lambda,\leq w}^{\mathfrak{p}_J} )^! \cong \Big( \bigoplus_{i \geq 0} \Ext_{\mathcal{\tilde O}_{\lambda,\leq w}^{\mathfrak{p}_J}}^i\big( \mathbf{L}^{\mathfrak{p}_J}_{\lambda,\leq w} , \mathbf{L}^{\mathfrak{p}_J}_{\lambda,\leq w} \big) \Big)^\mathrm{op} . \]
For $x \in X_\lambda^{J,\leq w}$, we further write $e^{\mathfrak{p}_J,\leq w}_{\lambda,x}$ for the idempotent in $A^{\mathfrak{p}_J}_{\lambda,\leq w}$ corresponding to the canonical projection $\mathrm{P}^{\mathfrak{p}_J}_{\lambda,\leq w} \to P^{\mathfrak{p}_J,\leq w}_{x\Cdot\lambda}$ and $f^{\mathfrak{p}_J,\leq w}_{\lambda,x}$ for the idempotent in the Koszul dual $( A^{\mathfrak{p}_J}_{\lambda,\leq w} )^!$ corresponding to the canonical projection $\mathrm{L}^{\mathfrak{p}_J}_{\lambda,\leq w} \to L_{x\Cdot\lambda}$.
Then, again by Theorem 3.12 in \cite{ShanVaragnoloVasserot}, the isomorphism
\[ A_{\mu,\leq w^{-1}}^{\mathfrak{p}_I} \xrightarrow{~\sim~} ( A_{\lambda,\leq w}^{\mathfrak{p}_J} )^! \]
from Theorem \ref{thm:KacMoodykoszuldual} can be chosen in such a way that it maps $e^{\mathfrak{p}_I,\leq w^{-1}}_{\mu,x^{-1}}$ to $f^{\mathfrak{p}_J,\leq w}_{\lambda,x}$ for all $x \in X_\lambda^J$.
In particular, if we write $L^{\mathfrak{p}_J,\leq w,!}_{\lambda,x}$ for the simple $( A^{\mathfrak{p}_J}_{\lambda,\leq w} )^!$-module corresponding to $x \in X_\lambda^{J,\leq w}$ (as described in Section \ref{sec:Koszul}) then the functor
\[ ( \mathcal{\tilde O}^{\mathfrak{p}_J}_{\lambda,\leq w} )^! \coloneqq ( A^{\mathfrak{p}_J}_{\lambda,\leq w} )^!\text{-}\mathrm{mod} \longrightarrow A^{\mathfrak{p}_I}_{\mu,\leq w^{-1}} \text{-} \mathrm{mod} \simeq \mathcal{\tilde O}^{\mathfrak{p}_I}_{\mu,\leq w^{-1}} \]
induced by the aforementioned isomorphism sends the simple $( A^{\mathfrak{p}_J}_{\lambda,\leq w} )^!$-module $L^{\mathfrak{p}_J,\leq w,!}_{\lambda,x}$ to the simple $\mathfrak{\tilde g}$-module $L_{x^{-1} \Cdot \mu}$.
We can now prove the main result of this section; it is an analogue of Theorem \ref{thm:extcategoryO} in the setting of affine Kac-Moody algebras.

\begin{Theorem} \label{thm:extKacMoody}
	Let $\lambda \in ( - \tilde \rho + \tilde X^+ ) \cup ( - \tilde \rho - \tilde X^+ )$ and $I,J \subsetneq \tilde S$ such that $\Stab_{\tilde W}(\lambda) = \tilde W_I$.
	If $\lambda$ has negative level then we have
	\[ \sum_{i \geq 0} \dim \Ext_{\mathcal{\tilde O}^{\mathfrak{p}_J}}^i( L_{y\Cdot\lambda} , \nabla^{\mathfrak{p}_J}_{x\Cdot\lambda} ) \cdot v^i = n^I_{x^{-1},y^{-1}} \]
	for all $x,y \in X_\lambda^J$, and if $\lambda$ has positive level then we have
	\[ \sum_{i \geq 0} \dim \Ext_{\mathcal{\tilde O}^{\mathfrak{p}_J}}^i( L_{y\Cdot\lambda} , \nabla^{\mathfrak{p}_J}_{x\Cdot\lambda} ) \cdot v^i = m_I^{w_Ix^{-1},w_Iy^{-1}} \]
	for all $x,y \in X_\lambda^J$.
\end{Theorem}
\begin{proof}
	Let us start by observing that by Theorems 3.42 and 3.59 in \cite{BrundanStroppel}, we have
	\[ \Ext_{\mathcal{\tilde O}^{\mathfrak{p}_J}}^i( L_{y\Cdot\lambda} , \nabla^{\mathfrak{p}_J}_{x\Cdot\lambda} ) \cong \Ext_{\mathcal{\tilde O}_{\lambda,\leq w}^{\mathfrak{p}_J}}^i( L_{y\Cdot\lambda} , \nabla^{\mathfrak{p}_J}_{x\Cdot\lambda} ) \]
	for all $i \geq 0$ and $w \in X_\lambda^J$ such that $x,y \leq w$; hence it suffices to prove the formulas in the theorem for the highest weight category $\mathcal{\tilde O}_{\lambda,\leq w}^{\mathfrak{p}_J}$.
	
	Next, for $x \in X_\lambda^{J,\leq w}$, let us write $\hat L^{\mathfrak{p}_J,\leq w}_{x\Cdot\lambda}$ and $\hat \Delta^{\mathfrak{p}_J,\leq w}_{x\Cdot\lambda}$ for the canonical graded lifts of $L_{x\Cdot\lambda}$ and $\Delta^{\mathfrak{p}_J}_{x\Cdot\lambda}$, respectively, in the category $\mathcal{\tilde O}_{\lambda,\leq w,\mathrm{gr}}^{\mathfrak{p}_J}$ of graded modules over a Koszul graded version $\hat A^{\mathfrak{p}_J}_{\lambda,\leq w}$ of $A^{\mathfrak{p}_J}_{\lambda,\leq w}$.
	Then, according to Lemma \ref{lem:abstractinversionformula}, the polynomials
	\[ p^{\mathfrak{p}_J,\leq w}_{\lambda,z,x} = \sum_{i \geq 0} \dim \Ext_{\mathcal{\tilde O}^{\mathfrak{p}_J}}^i( L_{x\Cdot\lambda} , \nabla^{\mathfrak{p}_J}_{z\Cdot \lambda} ) \cdot v^i \qquad \text{and} \qquad q^{\mathfrak{p}_J,\leq w}_{\lambda,z,y} = \sum_{i \geq 0} [ \hat \Delta^{\mathfrak{p}_J,\leq w}_{z\Cdot \lambda} : \hat L^{\mathfrak{p}_J,\leq w}_{y\Cdot\lambda} \langle i \rangle ] \cdot v^i \]
	with $x,y,z \in X_\lambda^{J,\leq w}$ satisfy the inversion formula
	\begin{equation} \label{eq:inversionKacMoody}
	\sum_{z \in X_\lambda^{J,\leq w}} p^{\mathfrak{p}_J, \leq w}_{\lambda,z,x}(-v) \cdot q^{\mathfrak{p}_J,\leq w}_{\lambda,z,y} = \delta_{x,y} ,
	\end{equation}
	and for $\mu \in ( - \tilde \rho + \tilde X^+ ) \cup ( - \tilde \rho - \tilde X^+ )$ with $\Stab_{\tilde W}(\mu) = \tilde W_J$ and such that $\mu$ has negative (or positive) level if $\lambda$ has positive (or negative) level, we have
	\begin{equation} \label{eq:KoszuldualKLpolynomialsKacMoody}
	p^{\mathfrak{p}_J,\leq w}_{\lambda,z,x} = q^{\mathfrak{p}_I,\leq w^{-1}}_{\mu,z^{-1},x^{-1}} \qquad \text{and} \qquad q^{\mathfrak{p}_J,\leq w}_{\lambda,z,y} = p^{\mathfrak{p}_I , \leq w^{-1}}_{\mu,z^{-1},y^{-1}}
	\end{equation}
	by Lemma \ref{lem:KoszuldualKLpolynomials} and the above discussion of Koszul duality for $\mathcal{\tilde O}^{\mathfrak{p}_J}_{\lambda,\leq w}$.
	
	As in the proof of Theorem \ref{thm:extcategoryO}, we first consider a special case of the formulas in the theorem, which we will then use to prove the general statement:
	Let $\nu \in - \tilde \rho - \tilde X^+$ with $\nu \neq - \tilde \rho$ and $\Stab_{\tilde W}(\nu) = \{ e \}$, so that $L_{w_J\Cdot\nu} = \Delta^{\mathfrak{p}_J}_{w_J\Cdot\nu}$ and
	\[ \sum_{i \geq 0} \dim \Ext_{ \mathcal{\tilde O}^{\mathfrak{p}_J} }^i( L_{w_J\Cdot\nu} , \nabla^{\mathfrak{p}_J}_{x\Cdot\nu} ) \cdot v^i = \delta_{x,w_J} = m^J_{w_Jx,e} \]
	for all $x \in X_\nu^J$ by Theorem 3.56 in \cite{BrundanStroppel}.
	
	Now one can show by induction on the length of $y \in X_\nu^J$ that the validity of the Kazhdan-Lusztig character formula in $\mathcal{\tilde O}_\nu$ (as conjectured in this case by G.\ Lusztig \cite{LusztigOnQuantumGroups} and proven by M.\ Kashiwara and T.\ Tanisaki \cite{KashiwaraTanisakiNegativeLevel}) and of its parabolic analogue in $\mathcal{\tilde O}_\nu^{\mathfrak{p}_J}$ (see Section 3.3 in \cite{LaniniRamSobajeFockSpace}) implies that
	\[ \sum_{i \geq 0} \dim \Ext_{\mathcal{\tilde O}}^i( L_{y\Cdot\nu} , \nabla^{\mathfrak{p}_J}_{x\Cdot\nu} ) \cdot v^i = m^J_{w_Jx,w_Jy} \]
	for all $x,y \in X_\nu^J$, exactly as in the proof of Proposition C.2 in \cite{Jantzen}.%
	\footnote{In order to replicate the proof of Proposition C.2 in \cite{Jantzen} in the category $\mathcal{\tilde O}_\nu^{\mathfrak{p}_J}$, one needs to use translation functors between $\mathcal{\tilde O}_\nu^{\mathfrak{p}_J}$ and $\mathcal{\tilde O}_\gamma^{\mathfrak{p}_J}$, where $\Stab_{\tilde W}(\gamma)$ is generated by a single simple reflection.
	These translation functors exist (and behave like they do in the setting of \cite{Jantzen}) by Section 3 in \cite{KashiwaraTanisakiNonCritical}; see also Proposition 4.36 in \cite{ShanVaragnoloVasserot}.}
	This implies that
	\[ p^{\mathfrak{p}_J,\leq w}_{\nu,z,x} = m^J_{w_Jz,w_Jx} \]
	for all $w \in X_\nu^J$ and $x,y \in X_\nu^{J,\leq w}$, and by equations \eqref{eq:inversionKacMoody} and \eqref{eq:inversionformulaparabolic}, we further have
	\[ q^{\mathfrak{p}_J,\leq w}_{\nu,x,y} = m_J^{w_Jx,w_Jy} . \]
	Now suppose that $\lambda$ has negative level.
	By Lemma 5.10 in \cite{ShanVaragnoloVasserot} (see also Remark 4.2 in \cite{KoCohomology}), the ordinary translation functor
	\[ T_\nu^\lambda \colon \mathcal{\tilde O}_{\nu,\leq w} \longrightarrow \mathcal{\tilde O}_{\lambda,\leq w} \]
	admits a graded lift
	\[ \hat T_\nu^\lambda \colon \mathcal{\tilde O}_{\nu,\leq w,\mathrm{gr}} \longrightarrow \mathcal{\tilde O}_{\lambda,\leq w,\mathrm{gr}} \]
	for all $w \in X_\lambda$,
	and the latter restricts to a functor from $\mathcal{\tilde O}_{\nu,\leq w,\mathrm{gr}}^{\mathfrak{p}_J}$ to $\mathcal{\tilde O}_{\lambda,\leq w,\mathrm{gr}}^{\mathfrak{p}_J}$ for $w \in X_\lambda^J$, as in the proof of Theorem \ref{thm:extcategoryO}.
	Still as in the proof of Theorem \ref{thm:extcategoryO}, we conclude that
	\[ q^{\mathfrak{p}_J,\leq w}_{\lambda,z,y} = q^{\mathfrak{p}_J,\leq w}_{\nu,z,y} = m_J^{w_Jz,w_Jy} \]
	for all $w \in X_\lambda^J$ and $y,z \in X_\lambda^{J,\leq w}$, and using the inversion formulas from \eqref{eq:inversionKacMoody} and Theorem \ref{thm:inversionformuladoublecosets}, it follows that
	\[ p^{\mathfrak{p}_J,\leq w}_{\lambda,z,x} = n^I_{z^{-1},x^{-1}} \]
	for all $w \in X_\lambda^J$ and $x,z \in X_\lambda^{J,\leq w}$, matching the first formula in the theorem.
	Furthermore, for a weight $\mu \in - \tilde \rho + \tilde X^+$ with $\mu \neq - \tilde \rho$ and $\Stab_{\tilde W}(\mu) = \tilde W_J$, equation \eqref{eq:KoszuldualKLpolynomialsKacMoody} yields
	\[ q^{\mathfrak{p}_I,\leq w}_{\mu,z,x} = p^{\mathfrak{p}_J,\leq w^{-1}}_{\lambda,z^{-1},x^{-1}} = n^I_{z,x} \qquad \text{and} \qquad p^{\mathfrak{p}_I,\leq w}_{\mu,z,y} = q^{\mathfrak{p}_J,\leq w^{-1}}_{\lambda,z^{-1},y^{-1}} = m_J^{w_Jz^{-1},w_Jy^{-1}} \]
	for all $w \in X_\mu^I$ and $x,y,z \leq X_\mu^{I,\leq w}$, and by relabeling, we obtain the second equation in the theorem.
\end{proof}

\begin{Remark}
	The first formula in Theorem \ref{thm:extKacMoody} is also obtained (using different methods) in the proof of Theorem 4.10 in \cite{KoCohomology}.
\end{Remark}

\appendix

\section{Coxeter groups and Kazhdan-Lusztig polynomials} \label{sec:appendix}

In this appendix, we summarize some important results about Kazhdan-Lusztig type combinatorics, and we prove the `double parabolic inversion formula' which was used in the proof of Theorem \ref{thm:extKacMoody}.
We follow the conventions of \cite{SoergelKL} and refer the reader to that article for more details and further references.
Let $W$ be a Coxeter group with a finite set of simple reflections $S \subseteq W$, so $W$ has a presentation of the form
\[ W = \langle s \in S \mid s^2 = e , (st)^{m_{st}} = e \text{ for } s,t \in S \rangle \]
for certain $m_{st} \in \{ 2,3,4,\ldots \} \cup \{ \infty \}$, where by convention, the relation $(st)^{m_{st}} = e$ is void if $m_{st} = \infty$.
We write $\ell \colon W \to \Z_{\geq 0}$ for the length function and $\leq$ for the Bruhat order on $W$.
For a subset $I \subseteq S$, we call $W_I \coloneqq \langle I \rangle$ the parabolic subgroup of $W$ corresponding to $I$.
For any $x \in W$, the coset $x W_I$ contains a unique element of minimal length, and the latter is also minimal in $x W_I$ with respect to the Bruhat order.
We write $W^I$ for the set of elements $x \in W$ that have minimal length in the coset $W^I$ and remark that for $x \in W^I$ and $w \in W_I$, we have $\ell(xw) = \ell(x) + \ell(w)$.
Analogously, we write $\prescript{I}{}{W}$ for the set of elements $y \in W$ that have minimal length in the coset $W_I y$ and remark that for $y \in \prescript{I}{}{W}$ and $w \in W_I$, we have $\ell(wy) = \ell(w) + \ell(y)$.
If $W_I$ is finite then we write $w_I$ for the longest element of $W_I$.
Then $W^I w_I$ is precisely the set of elements $x \in W$ that have maximal length in the coset $x W_I$ and $w_I \prescript{I}{}{W}$ is the set of elements $y \in W$ that have maximal length in the coset $W_I y$.

The \emph{Hecke algebra} of $W$ is the associative $\Z[v^{\pm 1}]$-algebra $\mathcal{H} = \mathcal{H}_W$ with generators $H_s$ for $s \in S$, subject to the \emph{quadratic relations}
\[ ( H_s + v )( H_s - v^{-1} ) = 0 \]
and the \emph{braid relations}
\[ \underbrace{ H_s H_t \cdots }_{m_{st}\text{ factors}} = \underbrace{ H_t H_s \cdots }_{m_{st}\text{ factors}} \]
for $s,t \in S$.
Using the quadratic relations, it is straightforward to see that there are two $\mathcal{H}$-modules \[ \mathrm{triv}_W = \Z[v^{\pm 1}] \qquad \text{and} \qquad \sign_W = \Z[v^{\pm 1}] \]
of rank one
such that $H_s$ acts on $\mathrm{triv}_W$ by multiplication with $v^{-1}$ and on $\sign_W$ by multiplication with $-v$ for all $s \in S$,
and using the braid relations, one can show that for $w \in W$ with a reduced expression $w = s_1 \cdots s_m$, the element $H_w \coloneqq H_{s_1} \cdots H_{s_m} \in \mathcal{H}$ is independent of the choice of reduced expression for $w$.
The elements $H_w$ with $w \in W$ are invertible and form a $\Z[v^{\pm1}]$-basis of $\mathcal{H}$.
We define the \emph{bar involution} of $\mathcal{H}$ to be the unique ring homomorphism $\overline{\phantom{A}} \colon \mathcal{H} \to \mathcal{H}$ with $\overline{v} = v^{-1}$ and $\overline{H_x} = H_{x^{-1}}^{-1}$, and we call an element $H \in \mathcal{H}$ \emph{self-dual} if $\overline{H} = H$.
For any $x \in W$, there exists a unique self-dual element $\underline{H}_x \in \mathcal{H}$ such that $\underline{H}_x \in H_x + \sum_{y<x} v \Z[v] \cdot H_y$, and the elements $\underline{H}_x$ with $x \in W$ form the \emph{Kazhdan-Lusztig basis} of $\mathcal{H}$.
The \emph{Kazhdan-Lusztig polynomial} $h_{y,x}$ corresponding to two elements $x,y \in W$ is defined by the equality
\[ \underline{H}_x = \sum_{y \in W} h_{y,x} \cdot H_y \]
and the \emph{inverse Kazhdan-Lusztig polynomial} $h^{y,x}$ is defined by
\[ H_y = \sum_{x \in W} (-1)^{\ell(y) + \ell(x)} \cdot h^{y,x} \cdot \underline{H}_x . \]

Now let us fix a subset $I \subseteq S$ and consider the \emph{spherical $\mathcal{H}$-module}
\[ \mathcal{M} = \mathcal{M}_I \coloneqq \mathrm{triv}_{W_I} \otimes_{\mathcal{H}_{W_I}} \mathcal{H} \]
and the \emph{anti-spherical $\mathcal{H}$-module}
\[ \mathcal{N} = \mathcal{N}_I \coloneqq \sign_{W_I} \otimes_{\mathcal{H}_{W_I}} \mathcal{H} . \]
The elements $M_x = 1 \otimes H_x \in \mathcal{M}$ and $N_x = 1 \otimes H_x \in \mathcal{N}$ for $x \in \prescript{I}{}{W}$ form bases of $\mathcal{M}$ and $\mathcal{N}$, respectively, and the bar involution on $\mathcal{H}$ induces involutions $\overline{\phantom{A}} \colon \mathcal{M} \to \mathcal{M}$ and $\overline{\phantom{A}} \colon \mathcal{N} \to \mathcal{N}$.
For any $x \in \prescript{I}{}{W}$, there is a unique self-dual element $\underline{M}_x \in \mathcal{M}$ and a unique self-dual element $\underline{N}_x \in \mathcal{N}$ with
\[ \underline{M}_x \in M_x + \sum_{y < x} v\Z[v] \cdot M_y \qquad \text{and} \qquad \underline{N}_x \in N_x + \sum_{y < x} v\Z[v] \cdot N_y , \]
and the elements $\underline{M}_x$ and $\underline{N}_x$ with $x \in \prescript{I}{}{W}$ form the \emph{Kazhdan-Lusztig bases} of $\mathcal{M}$ and $\mathcal{N}$, respectively.
For later use, we note that the right $\mathcal{H}$-module homomorphism
\[ \psi \colon \mathcal{H} \longrightarrow \mathcal{N} , \qquad H \longmapsto 1 \otimes H \]
satisfies $\psi(\underline{H}_x) = \underline{N}_x$ for all $x \in \prescript{I}{}{W}$ and $\psi( \underline{H}_x ) = 0$ for $x \in W \setminus \prescript{I}{}{W}$; see the proof of Proposition 3.4 in \cite{SoergelKL}.
The \emph{spherical Kazhdan-Lusztig polynomial} $m_{y,x}$ and the \emph{anti-spherical Kazhdan-Lusztig polynomial} $n_{y,x}$ corresponding to $x,y \in \prescript{I}{}{W}$ are defined by the equalities
\[ \underline{M}_x = \sum_{y \in \prescript{I}{}{W}} m_{y,x} \cdot M_y \qquad \text{and} \qquad \underline{N}_x = \sum_{y \in \prescript{I}{}{W}} n_{y,x} \cdot N_y , \]
and as before, we define the \emph{inverse spherical Kazhdan-Lusztig polynomial} $m^{y,x}$ and the \emph{inverse anti-spherical Kazhdan-Lusztig polynomial} $n^{y,x}$ by
\[ M_y = \sum_{x \in \prescript{I}{}{W}} (-1)^{\ell(y) + \ell(x)} \cdot m^{y,x} \cdot \underline{M}_x \qquad \text{and} \qquad N_y = \sum_{x \in \prescript{I}{}{W}} (-1)^{\ell(y) + \ell(x)} \cdot n^{y,x} \cdot \underline{N}_x . \]
As a consequence, we have the inversion formulas
\begin{equation} \label{eq:inversionformulaparabolic}
\sum_{z \in \prescript{I}{}{W}} (-1)^{\ell(z) + \ell(x)} \cdot m^{z,x} \cdot m_{z,y} = \delta_{x,y} \qquad \text{and} \qquad \sum_{z \in \prescript{I}{}{W}} (-1)^{\ell(z) + \ell(x)} \cdot n^{z,x} \cdot n_{z,y} = \delta_{x,y}
\end{equation}
for all $x,y \in \prescript{I}{}{W}$.
When the choice of $I \subseteq S$ is not clear from the context, we occasionally write
\[ m_{x,y} = m^I_{x,y} , \qquad n_{x,y} = n^I_{x,y} , \qquad m^{y,x} = m_I^{y,x} , \qquad n^{y,x} = n_I^{y,x} . \]
Now for $I \subseteq S$ such that $W_I$ is finite, consider the element
\[ 1_I \coloneqq \underline{H}_{w_I} = \sum_{w \in W_I} v^{\ell(w_I) - \ell(w)} \cdot H_w \]
and observe that we have
$H_s \cdot 1_I = 1_I \cdot H_s = v^{-1} \cdot 1_I$ for all $s \in I$; see Proposition 2.9 in \cite{SoergelKL}.
By the proof of Proposition 3.4 in \cite{SoergelKL}, the $\Z[v^{\pm 1}]$-linear map $\varphi \colon \mathcal{M} \to \mathcal{H}$ with $M_x \mapsto 1_I \cdot H_x$ for all $x \in \prescript{I}{}{W}$ is a homomorphism of right $\mathcal{H}$-modules with $\varphi(\underline{M}_x) = \underline{H}_{w_I x}$ for all $x \in \prescript{I}{}{W}$, and it induces an isomorphism between the right $\mathcal{H}$-modules $\mathcal{M}$ and $1_I \cdot \mathcal{H}$.
In particular, we have
\[ 1_I \cdot H_y = \sum_{z \in \prescript{I}{}{W}} (-1)^{\ell(y)+\ell(z)} \cdot m^{y,z} \cdot \underline{H}_{w_I z} \]
for all $y \in \prescript{I}{}{W}$.
Using the anti-involution $i$ from the proof of \cite[Theorem 2.7]{SoergelKL}, we further obtain
\[ H_y \cdot 1_I = \sum_{z \in W^I} (-1)^{\ell(y)+\ell(z)} \cdot m^{y^{-1},z^{-1}} \cdot \underline{H}_{z w_I} \]
and $\underline{H}_{xw_I} \in \mathcal{H} \cdot 1_I$ for all $x,y \in W^I$.
These observations will be useful later on.

In the following, we want to combine the `spherical' and the `anti-spherical' type Kazhdan-Lusztig combinatorics for two different parabolic subgroups of $W$ corresponding to subsets $I , J \subseteq S$ such that $W_I$ is finite.
To that end, we consider the $\Z[v^{\pm 1}]$-submodule $\mathcal{N}_J \cdot 1_I$ of $\mathcal{N}_J$, which by the above discussion can be considered as a generalization of both the anti-spherical $\mathcal{H}$-module $\mathcal{N}_J$ and the spherical $\mathcal{H}$-module $\mathcal{M}_I \cong 1_I \cdot \mathcal{H}$ (or strictly speaking, the analogous left $\mathcal{H}$-module $\mathcal{H} \cdot 1_I$).
This approach draws some inspiration from Subsection 2.2 in \cite{LaniniRamSobajeFockSpace}.
We first need to discuss some properties of parabolic double cosets.

Let us fix two subsets $I , J \subseteq S$.
The properties of the double cosets in $W_J \backslash W / W_I$ that we list in this paragraph follow from Exercise §1.3 in \cite[Chapter IV]{Bourbaki}; see also \cite[Subsection 2.2]{ParabolicDoubleCosets}.
For every element $x \in W$, the double coset $W_J x W_I$ has a unique element of minimal length, and the latter is also minimal in $W_J x W_I$ with respect to the Bruhat order.
The set of elements of $W$ that have minimal length in their $(W_J,W_I)$ double coset is precisely
\[ \prescript{J}{}{W}^I \coloneqq \prescript{J}{}{W} \cap W^I ; \]
hence $\prescript{J}{}{W}^I$ is a set of $(W_J,W_I)$ double coset representatives.
For $z \in \prescript{J}{}{W}^I$ and $H \coloneqq J \cap z I z^{-1}$, an element $x \in W_J$ satisfies $xz \in W^I$ if and only if $x \in W_J^H$, and the map $W_J^H \times W_I \to W_J z W_I$ with $(x,y) \mapsto xzy$ is a bijection with $\ell(xzy) = \ell(x) + \ell(z) + \ell(y)$.
Analogously, for $H^\prime \coloneqq z^{-1} J z \cap I$, an element $y \in W_I$ satisfies $zy \in \prescript{J}{}{W}$ if and only if $y \in \prescript{H^\prime}{}{W}_I$, and the map $W_J \times \prescript{H^\prime}{}{W}_I$ with $(x,y) \mapsto xzy$ is a bijection with $\ell(xzy) = \ell(x) + \ell(z) + \ell(y)$.
We say that the double coset $W_J z W_I$ (with $z \in \prescript{J}{}{W}^I$) is \emph{regular} if $J \cap z I z^{-1} = \varnothing$,
and we write
\[ \prescript{J}{}{W}^I_\mathrm{reg} = \{ z \in \prescript{J}{}{W}^I \mid W_J z W_I \text{ is regular} \} \]
for the set of minimal length representatives of the regular $(W_J,W_I)$-double cosets.

Now let $I,J \subseteq S$ such that $W_I$ is finite and consider the $\Z[v^{\pm 1}]$-submodule $\mathcal{N}_J \cdot 1_I$ of $\mathcal{N}_J$.
Observe that for $x \in \prescript{J}{}{W}^I$, $y \in W_J$ and $z \in W_I$ such that $\ell(yxz) = \ell(y) + \ell(x) + \ell(z)$, we have
\[ N_e \cdot H_{yxz} \cdot 1_I = N_e \cdot H_y H_x H_z \cdot 1_I = (-v)^{\ell(y)} \cdot v^{-\ell(z)} \cdot N_e \cdot H_x \cdot 1_I = (-1)^{\ell(y)} \cdot v^{\ell(y)-\ell(z)} \cdot N_x \cdot 1_I , \]
whence $\mathcal{N}_J \cdot 1_I$ is spanned over $\Z[v^{\pm 1}]$ by the elements $N_x \cdot 1_I$ with $x \in \prescript{J}{}{W}^I$.
If the double coset $W_J x W_I$ is non-regular then we can choose $s \in J \cap x I x^{-1}$ and $t = x^{-1} s x \in I$, and we compute
\begin{multline*}
	\qquad - v \cdot N_x \cdot 1_I = - v \cdot N_e \cdot H_x \cdot 1_I = N_e \cdot H_s \cdot H_x \cdot 1_I = N_e \cdot H_{sx} \cdot 1_I = N_e \cdot H_{xt} \cdot 1_I \\ = N_e \cdot H_x \cdot H_t \cdot 1_I = N_x \cdot H_t \cdot 1_I = v^{-1} \cdot N_x \cdot 1_I . \qquad 
\end{multline*}
Since $\mathcal{N}_J$ is free over $\Z[v^{\pm 1}]$, this implies that $N_x \cdot 1_I = 0$ and that $\mathcal{N}_J \cdot 1_I$ is spanned over $\Z[v^{\pm 1}]$ by the elements $N_x \cdot 1_I$ with $x \in \prescript{J}{}{W}^I_\mathrm{reg}$.
For $x \in \prescript{J}{}{W}^I_\mathrm{reg}$, we have
\[ N_x \cdot 1_I = \sum_{w \in W_I} v^{\ell(w_I) - \ell(w)} \cdot N_x \cdot H_w = \sum_{w \in W_I} v^{\ell(w_I) - \ell(w)} \cdot N_{xw} \]
because $xw \in \prescript{J}{}{W}$ for all $x \in W_I$.
Since the elements $N_x$ with $x \in \prescript{J}{}{W}$ form a $\Z[v^{\pm 1}]$-basis of $\mathcal{N}_J$, we conclude that the elements $N_x \cdot 1_I$ with $x \in \prescript{J}{}{W}^I_\mathrm{reg}$ form a $\Z[v^{\pm 1}]$-basis of $\mathcal{N}_J \cdot 1_I$.
Furthermore, still for $x \in \prescript{J}{}{W}^I_\mathrm{reg}$, we have $\underline{H}_{x w_I} \in \mathcal{H} \cdot 1_I$ (as observed above) and $\underline{N}_{xw_I} = \psi( \underline{H}_{xw_I} ) \in \mathcal{N}_J \cdot 1_I$.
Thus, the elements $\underline{N}_{x w_I}$ with $x \in \prescript{J}{}{W}^I_\mathrm{reg}$ form another basis of $\mathcal{N}_J \cdot 1_I$, which we call the \emph{Kazhdan-Lusztig basis}.
Now we can define polynomials $p_{y,x}$ and $p^{y,x}$ for $x,y \in \prescript{J}{}{W}^I_\mathrm{reg}$ by the equalities
\[ \underline{N}_{xw_I} = \sum_{ y \in \prescript{J}{}{W}^I_\mathrm{reg} } p_{y,x} \cdot N_y \cdot 1_I \qquad \text{and} \qquad N_y \cdot 1_I = \sum_{ x \in \prescript{J}{}{W}^I_\mathrm{reg} } (-1)^{\ell(y)+\ell(x)} \cdot p^{y,x} \cdot  \underline{N}_{xw_I} . \]
We could call these polynomials \emph{double parabolic Kazhdan-Lusztig polynomials}, although as we will presently observe, they match the usual (anti-spherical or inverse spherical) parabolic Kazhdan-Lusztig polynomials for a suitable choice of parameters.
Indeed, we have
\[ \sum_{z \in \prescript{J}{}{W}} n^J_{z,xw_I} \cdot N_z = \underline{N}_{xw_I} = \sum_{ y \in \prescript{J}{}{W}^I_\mathrm{reg} } p_{y,x} \cdot N_y \cdot 1_I = \sum_{ y \in \prescript{J}{}{W}^I_\mathrm{reg} } \sum_{w \in W_I} v^{\ell(w_I) - \ell(w)} \cdot p_{y,x} \cdot N_{yw} \]
and therefore $p_{y,x} = n^J_{yw_I,xw_I}$, and by applying the homomorphism $\psi\colon \mathcal{H} \to \mathcal{N}_J$ to the equality
\[ H_y \cdot 1_I = \sum_{z \in W^I} (-1)^{\ell(y)+\ell(z)} \cdot m_I^{y^{-1},z^{-1}} \cdot \underline{H}_{z w_I} , \]
we obtain
\[ \sum_{ \substack{ z \in W^I , \\ zw_I \in \prescript{J}{}{W} } } (-1)^{\ell(y)+\ell(z)} \cdot m_I^{y^{-1},z^{-1}} \cdot \underline{N}_{z w_I} = N_y \cdot 1_I = \sum_{ x \in \prescript{J}{}{W}^I_\mathrm{reg} } (-1)^{\ell(y)+\ell(x)} \cdot p^{y,x} \cdot  \underline{N}_{xw_I} \]
and therefore $p^{y,x} = m_I^{y^{-1},x^{-1}}$.
As an immediate consequence of these observations, we get the following double parabolic inversion formula:

\begin{Theorem} \label{thm:inversionformuladoublecosets}
	Let $I,J \subseteq S$ such that $W_I$ is finite.
	For $x,y \in \prescript{J}{}{W}^I_\mathrm{reg}$, we have
	\[ \sum_{ z \in \prescript{J}{}{W}^I_\mathrm{reg} } (-1)^{\ell(y)+\ell(z)} \cdot n^J_{zw_I,xw_I} \cdot m_I^{ z^{-1} , y^{-1} } = \delta_{x,y} . \]
\end{Theorem}
\begin{proof}
	This follows from the definition of the polynomials $p_{z,x}$ and $p^{z,y}$ and the fact that
	\[ p_{z,x} = n^J_{zw_I,xw_I} \qquad \text{and} \qquad p^{z,y} = m_I^{ z^{-1} , y^{-1} } , \]
	as observed above.
\end{proof}

\bibliographystyle{alpha}
\bibliography{cohomology}

\end{document}